\documentclass[11pt]{amsart}
\usepackage{hyperref}
\usepackage{anysize} \marginsize{1.3in}{1.3in}{1in}{1in}
\usepackage{amsfonts}

\usepackage{comment}
\usepackage[all,cmtip]{xy}
\usepackage{amsmath}
\usepackage{amsthm}
\usepackage{amssymb}
\usepackage{mathrsfs}
\usepackage{enumitem}
\usepackage{tikz-cd}
\usepackage{tikz}
\usepackage[normalem]{ulem}

\newtheorem{thm}{Theorem}

\newtheorem{prop}[thm]{Proposition}
\newtheorem{lemma}[thm]{Lemma}

\newtheorem{cor}[thm]{Corollary}

\theoremstyle{definition}

\newtheorem*{definition*}         {Definition}

\theoremstyle{remark}
\newtheorem{eg}[thm]{Example}
\newtheorem{remark}[thm]{Remark}

\newcommand{\Hom}{\mathrm{Hom}}

\newcommand{\Z}{\mathbb{Z}}

\newcommand{\C}{\mathbb{C}}

\def\reg{\mathrm{reg}}

\newcommand{\rk}{\mathrm{rk}\,}

\renewcommand{\phi}{\varphi}

\def\SO{\operatorname{SO}}

\def\Hom{\operatorname{Hom}}

\def\pt{\operatorname{pt}}

\def\rk{\operatorname{rk}}

\def\Pic{\operatorname{Pic}}

\def\gr{\operatorname{gr}}

\newcommand{\an}{\mathrm{an}}

\renewcommand{\bar}[1]{\overline{#1}}

\usepackage{amscd,amssymb,amsmath}
\usepackage{color}

\title{A short proof of a conjecture of Matsushita}

 \author[B. Bakker]{Benjamin Bakker}
\address{\noindent B. Bakker:  Dept. of Mathematics, Statistics, and Computer Science, University of Illinois at Chicago, Chicago, USA.}
\email{bakker.uic@gmail.com}

\def\O{\mathcal{O}}

\def\R{\mathbb{R}}
\def\Q{\mathbb{Q}}

\def\an{\mathrm{an}}
\def\dual{\vee}
\begin{document}
\maketitle
\begin{abstract}
    In this note we build on the arguments of van Geemen and Voisin \cite{vGV} to prove a conjecture of Matsushita that a Lagrangian fibration of an irreducible hyperk\"ahler manifold is either isotrivial or of maximal variation.  We also complete a partial result of Voisin \cite{voisin} regarding the density of torsion points of sections of Lagrangian fibrations. 
\end{abstract}
\def\opart{\circ}
\vskip2em
Let $X$ be an irreducible compact hyperk\"ahler manifold, that is, a simply-connected compact K\"ahler manifold $X$ for which $H^0(X,\Omega^2_X)=\C\sigma$ for a nowhere-degenerate holomorphic two-form $\sigma$.  A Lagrangian fibration of $X$ is a proper morphism $f:X\to B$ to a normal compact analytic variety $B$ whose generic fiber is smooth, connected, and Lagrangian (see \cite{LagFib} for a recent survey).  It follows that every smooth fiber is an abelian variety.  We let $B^\opart \subset B$ be a dense Zariski open smooth subset over which the base-change $f^\opart:X^\opart\to B^\opart$ is smooth.  By the period map of $f$ we mean the period map $\phi: B^\opart \to S$ to an appropriate moduli space $S$ of polarized abelian varieties associated to the natural variation of (polarized) weight one integral Hodge structures on $B^\opart$ with underlying local system $R^1f^\opart_*\Z_{X^\opart}$.  We say $f$ is \emph{isotrivial} if the period map is trivial (equivalently if $R^1f_*^\opart\Z_{X^\opart}$ has finite monodromy) and \emph{of maximal variation} if the period map is generically finite.   

Our main result is to resolve a conjecture of Matsushita:

\begin{thm}\label{main}
Let $X$ be an irreducible hyperk\"ahler manifold (or more generally a primitive symplectic variety in the sense of \cite{BLglobalmoduli}).  Then any Lagrangian fibration $f:X\to B$ is either isotrivial or of maximal variation. 
\end{thm}
Both possibilities in Theorem \ref{main} occur for K3 surfaces $S$---see for example \cite[Chapter 11]{Huybrechtsbook}---and therefore also for their Hilbert schemes $S^{[g]}$ in each (even) dimension.  Primitive symplectic varieties are the natural singular analog (as far as deformation theory is concerned) of irreducible hyperk\"ahler manifolds; see below for the definition and the precise meaning of a Lagrangian fibration in this context.  

Let $T_0\subset H^2(X,\Q)$ be the rational transcendental lattice, namely, the smallest rational Hodge substructure containing $[\sigma]\in H^{2,0}(X)$.  Theorem \ref{main} was proven by van Geemen and Voisin \cite[Theorem 5]{vGV} assuming $X$ is projective, that $T_0$ has generic (special) Mumford--Tate group (namely $\SO(T_0,q_X)$, where $q_X$ is the Beauville--Bogomolov--Fujiki form), and that $\rk T_0\geq 5$, by showing that under these conditions any fiber of a Lagrangian fibration that is not of maximal variation must be a factor of the Kuga--Satake variety of $T_0$.  Their result in particular applies to the generic projective deformation of $f:X\to B$ assuming $b_2(X)\geq 7$, which includes all known deformation types.  The Beauville--Mukai system of a nonprimitve basepoint-free ample divisor on a K3 surface has recently been shown to be of maximal variation in many cases by Dutta--Huybrechts \cite{duttahuybrechts} (see Corollary \ref{DuttaHuy} below) by understanding the derivative of the period map. 

We will instead prove Theorem \ref{main} by considering the complex variation of Hodge structures on $R^1f_*^\opart\C_{X^\opart}$.  We first recall the basic properties of complex variations.  A complex variation of Hodge structures on a Zariski open subset of a compact analytic variety (see for example \cite{DeligneMon}) consists of a $\C$-local system $V$ and a holomorphic (resp. antiholomorphic) descending filtration $F^\bullet$ (resp. $\bar F^\bullet$) such that we have a splitting of the sheaf of $C^\infty$ sections $A^0(V)=\bigoplus_p A^0(V^p)$ where $V^p=F^p\cap \bar F^{-p}$ and the flat connection maps $A^0(V^p)$ to $A^{1,0}(V^{p-1})\oplus A^1(V^p)\oplus A^
{0,1}(V^{p+1})$.  We refer to the grading $V^p$ as the Hodge grading and we say the level of the variation is the difference $p_{\mathrm{max}}-p_{\mathrm{min}}$ where $p_{\mathrm{max}}$ (resp. $p_{\mathrm{min}}$) is the maximum (resp. minimum) Hodge degree $p$ for which $V^p\neq 0$.  Observe that the level of a tensor product $V\otimes W$ is the sum of the levels of $V$ and $W$.  A polarization of the variation is a flat hermitian form $h$ for which the splitting is orthogonal and $(-1)^ph$ is positive definite on $V^p$.  In this case $\bar F^{-p}=(F^{p+1})^\perp$.  A variation which admits a polarization is said to be polarizable.  We define $\C(-d)$ to be the polarizable complex Hodge structure on $V=\C$ with $V^{d}=V$.  

Recall that the category of polarizable complex variations of Hodge structures is semi-simple.  The theorem of the fixed part \cite{Schmid} says that for two polarizable complex variations $V,W$, the group $\Hom(V,W)$ of morphisms of local systems has a natural complex Hodge structure whose degree $p$ part is exactly the morphisms of complex variations $V\to W(p)$.  We have the following further consequence due to Deligne:

\begin{thm}[{\cite[1.13 Proposition]{DeligneMon}}]\label{Deligne}Suppose $V$ is a $\C$-local system underlying a polarizable complex variation of Hodge structures and that we have a splitting of $\C$-local systems
    \begin{equation}\label{lssplit}V=\bigoplus_i N_i\otimes A_i\end{equation}
    where the $N_i$ are irreducible and pairwise non-isomorphic and the $A_i$ are nonzero complex vector spaces.  Then
\begin{enumerate}
\item Each $N_i$ underlies a polarizable complex variation of Hodge structures, unique up to shifting the Hodge grading.
\item Each polarizable complex variation of Hodge structures with underlying local system $V$ arises from \eqref{lssplit} by equipping each $N_i$ with its unique (up to shifts) polarizable complex variation of Hodge structures and each $A_i$ with a uniquely determined polarizable complex Hodge structure, namely $A_i=\Hom(N_i,V)$.
\end{enumerate}

\end{thm} 
In particular, the theorem implies a polarizable complex variation is irreducible if and only if the underlying local system is.  

\begin{remark}
The $\C$-local system $V$ underlying a polarizable complex variation of Hodge structures on an algebraic space is semi-simple.  This is a consequence of an unpublished result of Nori (see \cite[\S1.12]{DeligneMon}) and uses that the orthogonal complement of flat subbundle with respect to the hodge metric (which is a harmonic metric) is flat.  In our setting, we will only need to apply this to $V$ underlying a polarizable integral variation of Hodge structures.  In this case, the semi-simplicity of the underlying $\Q$-local system is a result of Deligne \cite[4.2.6]{DeligneHII}.  Note that for a perfect field $K$ and a field extension $K\subset L$, a $K$-local system $V$ is semi-simple if and only if $V_L$ is, as semi-simplicity is equivalent to being generated by simple sub-local systems. 
\end{remark}

Given an $\R$-local system $V$, a polarizable real variation of Hodge structures\footnote{Throughout, by a real variation we mean a pure real variation, unless otherwise specified.} on $V$ in the usual sense naturally induces a polarizable complex variation on $V_\C$.  Conversely, a polarizable complex variation on $V_\C$ comes from a polarizable real variation on $V$ if complex conjugation flips the Hodge grading, or more precisely if for some (hence any) polarization $h$ the isomorphism of local systems $V_\C\to V_\C^\vee$ given by $y\mapsto h(-,\bar y)$ induces an isomorphism of complex variations $V_\C\to V_\C^\vee(-w)$ for some (uniquely determined) $w$.  Indeed, if this is the case then $V^p\xrightarrow{\cong}(V^{w-p})^\vee$ so $\bar {V^p}=V^{w-p}$.  Moreover, for even $w$ (resp. odd $w$) a real polarization is provided by the symmetric (resp. antisymmetric) real form $q(x,y)=h(x,\bar y)+h(y,\bar x)$ (resp. $q(x,y)=i(h(x,\bar y)-h(y,\bar x))$), since $q(x,\bar x)=h(x,x)+h(\bar x,\bar x)$ (resp. $-iq(x,\bar x)=h(x,x)-h(\bar x,\bar x)$) is definite of alternating sign on $V^p$.  

The category of polarizable real variations is also semi-simple.  Observe that by Theorem \ref{Deligne} any isotypic component $W$ of an $\R$-local system $V$ underlying a polarizable real variation is canonically a real subvariation, as the same is true over $\C$ and the isomorphism $V_\C\to V_\C^\vee(-w)$ coming from $h$ restricts to an isomorphism $W_\C\to W_\C^\vee(-w)$.  If $V$ is a single isotypic factor, then $V_\C$ either has one self-conjugate irreducible factor $N$ or has two non-isomorphic conjugate irreducible factors $N,\bar N$.  Note that $N^\dual\cong \bar N$ via the polarization.  Note also that the level of a polarizable real variation $V$ is at least as large as the level of any of the irreducible factors of $V_\C$.

We say that a real or complex variation is \emph{isotrivial} if the Hodge filtration is flat, or equivalently if the irreducible factors of the complexification are level zero\footnote{Or equivalently, if the monodromy is unitary (by Theorem \ref{Deligne}); since there may not be an integral structure, this does not necessarily mean the monodromy is finite.}.  To summarize the above discussion:

\begin{lemma}\label{alternative}Let $V$ be an irreducible polarizable real variation of Hodge structures of level one.  Then $V$ is either isotrivial, or every irreducible factor of $V_\C$ is level one.
\end{lemma}

Before turning to the proof of Theorem \ref{main} we recall the definition of a primitive symplectic variety.  Let $X$ be a symplectic variety in the sense of Beauville\footnote{This definition is equivalent to Beauville's original one.} \cite{beauvillesing}, that is, a compact K\"ahler variety with rational singularities and a nowhere degenerate 2-form $\sigma$ on its regular locus $X^\reg$.  We say that $X$ is primitive symplectic if $H^1(X,\O_X)=0$ and $H^0(X^\reg,\Omega^2_{X^\reg})=\C\sigma$.  As the singularities are rational, for any resolution $\pi:Y\to X$ the form $\sigma$ extends to a two-form on $Y$ \cite[Corollary 1.7]{KebekusSchnell}.  Moreover, $\pi^*:H^2(X,\Q)\to H^2(Y,\Q)$ is injective, so the Hodge structure on $H^2(X,\Q)$ is pure, and we have an induced isomorphism $\pi^*:H^{2,0}(X)\to H^{2,0}(Y)$ (see \cite{BLglobalmoduli} for details).  In particular we have a well defined class $[\sigma]\in H^{2,0}(X)$.  

By a Lagrangian fibration of a primitive symplectic variety we still mean a proper morphism $f:X\to B$ to a normal compact analytic variety $B$ whose generic fiber is smooth, connected, and Lagrangian.  Each smooth fiber will still be an abelian variety.  Moreover, $B$ is in fact K\"ahler and Moishezon \cite{varoouchas} (and in particular an algebraic space) since $f$ is equidimensional as in \cite[Lemma 1.17]{LagFib}, now using functorial pullback of reflexive forms \cite[Theoreom 1.11]{KebekusSchnell} and the fact that $R\pi_*\omega_Y\cong\omega_X\cong \O_X$ by the rationality of the singularities of $X$ \cite[\S5.1]{kollarmori}.

We use the same notation as above:  $B^\opart \subset B$ is a dense Zariski open smooth subset over which the restriction $f^\opart:X^\opart\to B^\opart$ is smooth and $\phi: B^\opart \to S$ is the period map associated to the variation of (polarized) weight one integral Hodge structures on $B^\opart$ with underlying local system $R^1f^\opart_*\Z_{X^\opart}$.

\begin{proof}[Proof of Theorem \ref{main}]

Let $V_\Z:=R^1f^\opart_* \Z_{X^\opart}$.  We start with the following result of Voisin, whose proof we give for convenience (and to extend it slightly).

\begin{lemma}[{\cite[Lemma 5.5]{voisin}}]\label{irred} $V_\R$ is irreducible as a polarizable real variation of Hodge structures.
\end{lemma}
\begin{proof}
First assume $X$ is smooth.  By a result of Matsushita \cite[Lemma 2.2]{matsushitaL} the restriction map $H^2(X,\R)\to H^2(X_b,\R)$ to a generic fiber of $f^\opart$ is rank one and by Deligne's global invariant cycles theorem $H^2(X,\R)\to H^0(B^\opart,R^2f^\opart_*\R_{X^\opart})$ is surjective \cite{DeligneHII}.  If $V_\R$ splits as a variation then the polarizations of the factors would yield a larger than one-dimensional space of sections of $R^2f^\opart_*\R_{X^\opart}=\wedge^2 V_\R$, which is a contradiction.

Now if $X$ is a primitive symplectic variety, one easily checks using the results of \cite{BLglobalmoduli} that Matsushita's proof carries through verbatim and that $H^2(X,\R)\to H^0(B^\opart,R^2f^\opart_*\R_{X^\opart})$ is still surjective, since the cokernel of $\pi^*:H^2(X,\R)\to H^2(Y,\R)$ is generated by exceptional divisors for a log resolution $\pi:Y\to X$ since $X$ has rational singularities.
\end{proof}

Suppose now that $f$ is not of maximal variation.  Define the real transcendental lattice $T\subset H^2(X,\R)$ to be the polarizable real Hodge substructure spanned by $[\sigma]$ and $[\bar{\sigma}]$.  We next claim that the polarizable real variation of Hodge structures $V_\R\otimes T^\vee$ has a nontrivial subvariation of level at most one after a finite base-change; the argument below is that of \cite{vGV} at its core, with some mild modifications. 

Let $\nu:B^{\opart\prime}\to B^\opart$ be a finite Galois \'etale cover so that the base-change $V'_\Z:=\nu^*V_\Z$ is pulled back along its period map $\phi' :B^{\opart\prime}\to S'$, where $S'$ is a level cover of $S$.  Note that up to replacing $B^{\opart\prime}$ with a further finite cover, we may assume $\phi'$ can be embedded in a proper map $\bar \phi':\bar {B}^{\opart\prime}\to S'$ \cite{Giii}.  Denote by $Z\subset S'$ the image of $\bar\phi'$, by $\psi:B^{\opart\prime}\to Z$ the resulting map, and by $V''_\Z$ the variation on $Z$ so that $V'_\Z=\psi^*V_\Z''$.    The map $\bar\phi'$ and its image $Z$ are in fact algebraic \cite[Theorem 3.1]{borel}.  We shrink $Z$ (and $B^\opart,B^{\opart\prime},X^\opart$) so that it is smooth and so that $R^1\psi_*\R_{B^{\opart\prime}}$ is a local system, naturally underlying a graded polarizable real variation of mixed Hodge structures whose only nonzero Hodge components are $(0,0),(1,0),(0,1),(1,1)$ (for example using Saito's theory of mixed Hodge modules \cite{saito88,saito90}).  Let $f^{\opart\prime}:X^{\opart\prime}\to B^{\opart\prime}$ be the base-change of $f$.  The natural map $H^2(X,\C)\to H^0(B^{\opart\prime},R^2f^{\opart\prime}_*\C_{X^{\opart\prime}})$ sends $[\sigma]$ and hence $T_\C$ to zero, since the fibers of $f$ are Lagrangian.  By the Leray spectral sequence we have a natural morphism $\pt_Z^*T\to R^1\psi_*V_\R'\cong V_\R''\otimes R^1\psi_*\R_{B^{\opart\prime}}$ in the category of real variations of mixed Hodge structures.  This map is nonzero from the following geometric description as in \cite{vGV}.  

Through a very general point $b\in B^{\opart\prime}$, say above a point $z\in Z$, let $F$ be the positive-dimensional fiber of $\psi$ through $b$.  The restricted family $X_F\to F$ has trivial monodromy so the following natural diagram commutes
\[\xymatrix{
T_\C\ar[r]\ar[d]&(V''_\C\otimes R^1\psi_*\C_{B^{\opart\prime}})_z\ar[d]\\
H^2(X_{F},\C)\ar[r]&H^1(X_b,\C)\otimes H^1(F,\C)
}\]
where the bottom arrow comes from the degeneration of the Leray spectral sequence for $X_F\to F$.  In the projective case we have $X_F\cong X_b\times F$ (possibly after a further base-change) and this map is just the K\"unneth projection.  The image of $[\sigma]$ is nonzero in the bottom right corner since: (i) $\sigma$ is nonzero when restricted to $X_F$ since $\dim X_F>\frac{1}{2}\dim X$; (ii) $\sigma|_{X_F}$ extends to a smooth compactification since $\sigma$ extends to a smooth compactification of $X^\opart$, so $[\sigma]\neq0\in H^2(X_F,\C)$; (iii) the image of $[\sigma]$ in $H^2(X_b,\C)\otimes H^0(F,\C)$ vanishes and $[\sigma]$ is not in the image of $ H^0(X_b,\C)\otimes H^2(F,\C)$, as it is not pulled back from $F$.

Thus, there is a nonzero morphism of real variations 
\begin{equation}\label{image} \gr_{-1}^W\psi^*(R^1\psi_*\R_{B^{\opart\prime}})^\vee\to V_\R'\otimes T^\vee.\end{equation}
As the category of polarizable real variations of (pure) Hodge structures is semi-simple, we therefore have a splitting
\begin{equation}\label{split}\notag V_\R'\otimes T^\vee=U\oplus W\end{equation}
of real variations, where $U\neq 0$ is the image of \eqref{image}.  In particular, $U$ has level at most one and weight -1.

Now by Lemma \ref{irred} the Galois group of $\nu$ acts transitively on the isotypic factors of $V'_\R$.  In particular, if $f$ (and therefore $V_\R$) is not isotrivial, no factor of $V'_\R$ (as a variation) is isotrivial, or else its entire isotypic component would be, and so would $V_\R$.  But then there can be no nonzero morphism of variations $V'_\R\otimes T^\vee\to U$.  Indeed, by Lemma \ref{alternative}, an irreducible factor $N$ of $V'_\C$ has level one of degrees 0,1, and $N\otimes T^\vee_\C$ can only map nontrivially to an irreducible factor of $U_\C$ of the form $N(1)$, while $\Hom(N\otimes T^\vee_\C,N(1))=T_\C(1)\cong\C(-1)\oplus\C(1)$ has no degree 0 elements.  Thus, $f$ must be isotrivial.
\end{proof}

\begin{eg}
We revisit the example from \cite[\S4]{vGV}.  Let $p\geq 5$ be a prime and $\lambda$ a $p$th root of unity.  Consider a family of abelian varieties $f:X\to B$ with a cyclic automorphism such that the induced automorphism $\alpha$ of $V_\R=R^1f_*\R_{X}$ has $\lambda$ as an eigenvalue on $V^{1,0}$ but not on $V^{0,1}$.  Let $\alpha'$ be the automorphism of $T^\vee$ with eigenvalue $\lambda^{-1}$ on $(T^\vee)^{-2,0}$ and eigenvalue $\lambda$ on $(T^\vee)^{0,-2}$.  Then $V_\R\otimes T^\vee$ has a level one factor, namely the $1$ eigenspace $(V_\R\otimes T^\vee)^1$ of $\alpha\otimes \alpha'$.  But the condition on the eigenvalues means the eigenspaces $(V_\C)^\lambda$ and $(V_\C)^{\lambda^{-1}}$ are level zero, and the real variation $(V_\C)^\lambda\oplus(V_\C)^{\lambda^{-1}}$ is an isotrivial real factor.
\end{eg}

We now discuss a few applications of Theorem \ref{main}.  Dutta--Huybrechts have proven many cases \cite[Theorem 1.1]{duttahuybrechts} of the following corollary, and have also shown that Theorem \ref{main} implies the general case:
\begin{cor}\label{DuttaHuy}
Let $H$ be a basepoint-free ample divisor on a K3 surface $S$.  Then the complete linear system $|H|$ is of maximal variation.
\end{cor}
\begin{proof}
The genus 2 case is proven unconditionally in \cite[Prop. 5.4]{duttahuybrechts}, and the genus $g\geq 3$ case in \cite[Prop. 5.2]{duttahuybrechts} assuming Theorem \ref{main}.
\end{proof}

The next application uses a result of Gao, which is a simple application of the Ax--Schanuel theorem for universal families of abelian varieties \cite[Theorem 1.1]{GaoAB}.  Recall that for a projective family $f:X\to B$ of $g$-dimensional abelian varieties\footnote{Meaning $X$ is an abelian scheme over $B$, so there is in particular a 0-section of $f$.} equipped with a section $s$ and letting $
\tilde B\to B^\an$ be the universal cover, the Betti map $\beta: \tilde B\to H_1(X_b,\R)$ is the real analytic map obtained by taking the coordinates of the section $s$ with respect to the flat real-analytic trivialization of $f$.  Observe that $\beta^{-1}(H_1(X_b,\Q))$ is the set of points of $\tilde B$ at which $s$ is torsion.  If $\phi: B\to S$ is the period map of $f$ and $\mathcal{X}\to S$ the universal family of abelian varieties, then $s$ naturally yields a map $B\to \mathcal{X}$ lifting $\phi$.  We say that $s$ is of maximal variation if $B\to \mathcal{X}$ is generically finite.

\begin{prop}[{\cite[Theorem 9.1]{GaoAB}}]\label{AS}
Let $f:X\to B$ be a projective family of $g$-dimensional abelian varieties with $\dim B\geq g$ and whose very general fiber has no nontrivial $\Q$-factor.  Let $s:B\to X$ be a non-torsion section of $f$ which is of maximal variation.  Then the Betti map $\beta:\tilde B\to H_1(X_b,\R)$ associated to $s$ is generically submersive.
\end{prop}

\begin{cor}\label{cor}
Let $X$ be a primitive symplectic variety and $f:X\to B$ a Lagrangian fibration.  Let $L$ be a line bundle whose restriction to the smooth fibers is topologically trivial.  Then the set of points $b\in B^\opart(\C)$ for which $L|_{X_b}$ is torsion is analytically dense in $B$.  
\end{cor}
Corollary \ref{cor} was proven by Voisin \cite[Theorem 1.3]{voisin} assuming either $f$ is of maximal variation and $\dim X\leq 8$ or isotrivial with no restriction on the dimension.  Our use of Proposition \ref{AS} replaces the results of Andr\'e--Corvaja--Zannier \cite{ACZ} in \cite{voisin}.

\begin{proof}[Proof of Corollary \ref{cor}]  By Voisin's result and Theorem \ref{main} we may assume $f$ is of maximal variation.  Note that Voisin's proof works equally well in the singular case; we leave the details to the reader.  Consider the family of abelian varieties $h:\Pic^0(X^\opart/B^\opart)\to B^\opart$ and the section $s:b\mapsto L|_{X_b}$.  Let $\nu:B^{\opart\prime}\to B^\opart$ be a Galois finite base-change for which the $\Q$-factors of the very general fiber of $h$ are defined over $B^{\opart\prime}$.  As the Galois group of $\nu$ acts transitively on the factors by Lemma \ref{irred}, the $d$ factors all have the same dimension $g'$, and the image of the period map of each factor must have dimension $\geq g'$, or else the image of the period map of $f$ would have dimension smaller than $dg'=\dim(X^\opart/B^\opart)=\dim(B^\opart)$.  The base-change of the section $s$ is also Galois invariant, so it suffices to prove the density statement for its projection to a single factor $Y^\opart\to B^{\opart\prime}$.  Applying Proposition \ref{AS}, the Betti map $\beta:\tilde B\to H_1(Y_b,\R)$ is submersive, so $\beta^{-1}(H_1(Y_b,\Q))$ is analytically dense in $\tilde B$ as claimed. 
\end{proof}

\def\CH{\operatorname{CH}}

Corollary \ref{cor} has an interesting interpretation in terms of the Beauville conjecture for irreducible hyperk\"ahler manifolds $X$, see the discussion in \cite[\S 1.2]{voisin}.  There it is also shown how corollaries \ref{DuttaHuy} and \ref{cor} can be used to construct constant cycle curves on K3 surfaces.

\subsection*{Acknowledgements}The author would like to thank C. Voisin for conversations related to Theorem \ref{main} and Corollary \ref{cor}, D. Maulik for bringing the paper \cite{voisin} to his attention, and C. Lehn for conversations related to Lagrangian fibrations in the singular case.  This work was partially supported by NSF grant DMS-2131688.

\bibliography{biblio.voisin}
\bibliographystyle{plain}

\end{document}